\documentclass{article}
\usepackage{amsmath}
\usepackage{amssymb}
\usepackage{amsthm}
\usepackage{bbm,mathtools}
\usepackage{stmaryrd}

\newcommand{\op}[1]{\prescript{o}{}{#1}}
\newcommand{\pp}[1]{\prescript{p}{}{#1}}

\newcommand{\one}{\mathbbm 1}

\renewcommand\vec[1]{\overset{{}_{\shortrightarrow}}{#1}}
\newcommand\cev[1]{\overset{{}_{\shortleftarrow}}{#1}}

\def\reals{\mathbb{R}}

\def\comp{\raise 1pt \hbox{$\scriptstyle\circ$}}

\def\maximize{\mathop{\rm maximize}\limits}
\def\esssup{\mathop{\rm ess\ sup}\nolimits}

\def\upto{{\raise 1pt \hbox{$\scriptstyle \,\nearrow\,$}}}
\def\downto{{\raise 1pt \hbox{$\scriptstyle \,\searrow\,$}}}

\def\FF{(\F_t)_{t\ge 0}}
\def\ovr{\mathop{\rm over}}

\def\C{{\cal C}}

\def\D{{\cal D}}

\def\F{{\cal F}}

\def\J{{\cal J}}

\def\M{{\cal M}}
\def\N{{\cal N}}

\def\R{{\cal R}}

\def\T{{\cal T}}

\newtheorem{theorem}{Theorem}
\newtheorem{lemma}[theorem]{Lemma}

\theoremstyle{definition}

\theoremstyle{empty}

\begin{document}
\title{Optimal stopping without Snell envelopes}

\author{Teemu Pennanen \and Ari-Pekka Perkki\"o}

\maketitle

\begin{abstract}
This paper proves the existence of optimal stopping times via elementary functional analytic arguments. The problem is first relaxed into a convex optimization problem over a closed convex subset of the unit ball of the dual of a Banach space. The existence of optimal solutions then follows from the Banach--Alaoglu compactness theorem and the Krein--Millman theorem on extreme points of convex sets. This approach seems to give the most general existence results known to date. Applying convex duality to the relaxed problem gives a dual problem and optimality conditions in terms of martingales that dominate the reward process.
\end{abstract}


\section{Introduction}

Given a complete filtered probability space $(\Omega,\F,\FF,P)$ satisfying the usual hypotheses, let $R$ be an optional process of class $(D)$, and consider the optimal stopping problem
\begin{equation}\label{os}\tag{OS}
\maximize\quad ER_\tau\quad\ovr\quad \tau\in\T,
\end{equation}
where $\T$ is the set of stopping times with values in $[0,T]\cup\{T+\}$ and $R$ is defined to be zero on $T+$. We allow $T$ to be $\infty$ in which case $[0,T]$ is interpreted as the one-point compactification of the positive reals.

Without further conditions, optimal stopping times need not exist (take any deterministic process $R$ whose supremum is not attained). 
Theorem~II.2 of Bismut and Skalli~\cite{bs77} establishes the existence for bounded reward processes $R$ such that $R\ge\cev R$ and $\vec R\le\pp R$. Here,
\[
\vec R_t:=\limsup_{s\upto t} R_s\quad\text{and}\quad\cev R_t:=\limsup_{s\downto t} R_s,
\]
the {\em left-} and {\em right-upper semicontinuous regularizations} of $R$, respectively. Bismut and Skalli mention on page 301 that, instead of boundedness, it would suffice to assume that $R$ is of class $(D)$. 

In order to extend the above, we study the ``optimal quasi-stopping problem''
\begin{equation}\label{oqs}\tag{OQS}
\maximize\quad E[R_\tau+\vec R_{\tilde \tau}]\quad\ovr\quad (\tau,\tilde\tau)\in\hat\T,
\end{equation}
where $\hat\T$ is the set of {\em quasi-stopping times} (``split stopping time'' in Dellacherie and Meyer~\cite{dm82}) defined by
\[
\hat\T:=\{(\tau,\tilde\tau)\in\T\times\T_p \mid \tilde\tau>0,\ \tau\vee\tilde\tau =T+ \},
\]
where $\T_p$ is the set of predictable times. When $R$ is cadlag, $\vec R=R_-$, and our formulation of the quasi-optimal stopping coincides with that of Bismut~\cite{bis79}. Our main result gives the existence of optimal quasi-stopping times when $R\ge\cev R$. When $R \ge \cev R$ and $\vec R\le \pp R$, we obtain the existence for \eqref{os} thus extending the existence result of \cite[Theorem~II.2]{bs77} to possibly unbounded processes $R$ as suggested already on page 301 of \cite{bs77}. 

Our existence proofs are based on functional analytical arguments that avoid the use of Snell envelopes which are used in most analyses of optimal stopping. Our strategy is to first look at a convex relaxation of the problem. This turns out be a linear optimization problem over a compact convex set of random measures whose extremal points can be identified with (quasi-)stopping times. As soon as the objective is upper semicontinuous on this set, Krein-Milman theorem gives the existence of (quasi-)stopping times. Sufficient conditions for upper semicontinuity are obtained as a simple application of the main result of Perkki\"o and Trevino~\cite{pt18a}. The overall approach was suggested already on page 287 of Bismut~\cite{bis79b} in the case of optimal stopping. We extended the strategy (and provide explicit derivations) to quasi-optimal stopping for a merely right-upper semicontinuous reward process.

The last section of the paper develops a dual problem and optimality conditions for optimal (quasi-)stopping problems. The dual variables turn out to be martingales that dominate $R$. As a simple consequence, we obtain the duality result of Davis and Karatzas~\cite{dk94} in a more general setting where the reward process $R$ is merely of class $(D)$.

\section{Regular processes}\label{sec:regular}

In this section, the reward process $R$ is assumed to be {\em regular}, i.e.\ of class $(D)$ such that the left-continuous version $R_-$ and the predictable projection $\pp R$ of $R$ are indistinguishable; see e.g.\ \cite{bis78} or \cite[Remark~50.d]{dm82}. Our analysis will be based on the fact that the space of regular processes is a Banach space whose dual can be identified with optional measures of essentially bounded variation; see Theorem~\ref{thm:reg1} below.


The space $M$ of Radon measures may be identified with the space $X_0$ of left-continuous functions of bounded variation on $\reals_+$ which are constant on $(T,\infty]$ and $x_0=0$. Indeed, for every $x\in X_0$, there exists a unique $Dx\in M$ such that $x_t=Dx([0,t))$ for all $t\in\reals$. Thus $x\mapsto Dx$ defines a linear isomorphism between $X_0$ and $M$. The value of $x$ for $t>T$ will be denoted by $x_{T+}$. Similarly, the space $\M^\infty$ of optional random measures with essentially bounded total variation may be identified with the space $\N_0^\infty$ of adapted processes $x$ with $x\in X_0$ almost surely and $Dx\in\M^\infty$.

Let $C$ the space of continuous functions on $[0,T]$ equipped with the supremum norm and let $L^1(C)$ be the space of (not necessarily adapted) continuous processes $y$ with $E\|y\|<\infty$. The norm $E\|y\|$ makes $L^1(C)$ into a Banach space whose dual can be identified with the space $L^\infty(M)$ of random measures whose pathwise total variation is essentially bounded. The following result is essentially from \cite{bis78}; see \cite[Theorem~8]{pp18c} or \cite[Corollary~16]{pp18b}. It provides the functional analytic setting for analyzing optimal stopping with regular processes.

\begin{theorem}\label{thm:reg1}
The space $\R^1$ of regular processes equipped with the norm
\[
\|y\|_{\R^1}:=\sup_{\tau\in\T}E|y_\tau|
\]
is Banach and its dual can be identified with $\M^\infty$ through the bilinear form
\[
\langle y,u\rangle =E\int ydu.
\]
The optional projection is a continuous surjection of $L^1(C)$ to $\R^1$ and its adjoint is the embedding of $\M^\infty$ to $L^\infty(M)$. The norm of $\R^1$ is equivalent to
\[
p(y):= \inf_{z\in L^1(C)} \{E\|z\| \mid \op z = y\}
\]
which has the dual representation
\[
p(y)=\sup\{\langle y,u\rangle\,|\,\esssup(\|u\|)\le 1\}.
\]
\end{theorem}

We first write the optimal stopping problem as
\[
\maximize\quad \langle R, Dx\rangle\quad\ovr\quad x\in\C_e,
\]
where 
\[
\C_e:=\{x\in\N_0^\infty\,|\, Dx\in\M^\infty_+,\ x_t\in\{0,1\}\}.
\]
The equation $\tau(\omega) = \inf\{t\in\reals\mid x_t(\omega)\ge 1\}$ gives a one-to-one correspondence between the elements of $\T$ and $\C_e$. Consider also the convex relaxation
\begin{equation}\label{Rcr}\tag{ROS}
\maximize\quad \langle R,Dx\rangle\quad\ovr\quad x\in\C,
\end{equation}
where
\[
\C:=\{x\in\N_0^\infty\,|\, Dx\in\M^\infty_+,\ x_{T+}\le 1\}.
\]
Clearly, $\C_e\subset\C$ so the optimum value of optimal stopping is dominated by the optimum value of the relaxation. The elements of $\C$ are {\em randomized stopping times} in the sense of Baxter and Chacon~\cite[Section~2]{bc77}. 

Recall that $x\in\C$ is an {\em extreme point} of $\C$ if it cannot be expressed as a convex combination of two points of $\C$ different from $x$.

\begin{lemma}\label{lem:km}
The set $\C$ is convex, $\sigma(\N_0^\infty,\R^1)$-compact and $\C_e$ is the set of its extreme points.
\end{lemma}

\begin{proof}
The set $\C$ is a closed convex set of the unit ball that $\N_0^\infty$ has as the dual of the Banach space $\R^1$. The compactness thus follows from Banach-Alaoglu. It is easily shown that the elements of $\C_e$ are extreme points of $\C$. On the other hand, if $x\notin\C_e$ there exists an $\bar s\in(0,1)$ such that the processes
\[
x^1_t:=\frac{1}{\bar s}[x_t\wedge\bar s]\quad\text{and}\quad x^2_t:=\frac{1}{1-\bar s}[(x_t-\bar s)\vee 0]
\]
are different elements of $\C$. Since $x=\bar s x^1+(1-\bar s)x^2$, it is not an extreme point of $\C$.
\end{proof}

Since the function $x\mapsto\langle R,Dx\rangle$ is continuous, the compactness of $\C$ in Lemma~\ref{lem:km} implies that the maximum in \eqref{Rcr} is attained. The fact that the maximum is attained at a genuine stopping time follows from the characterization of the extreme points in Lemma~\ref{lem:km} and the following variant of the Krein-Millman theorem; see e.g.~\cite[Theorem~25.9]{cho69}. 

\begin{theorem}[Bauer's maximum principle]\label{thm:bauer}
In a locally convex Hausdorff topological vector space, an upper semicontinuous (usc) convex function on a compact convex set $K$ attains its maximum at an extremal point of $K$.
\end{theorem}


Combining Lemma~\ref{lem:km} and Theorem~\ref{thm:bauer} gives the following.

\begin{theorem}\label{thm:os}
Optimal stopping time in \eqref{os} exists for every $R\in\R^1$.
\end{theorem}

The above seems to have been first proved in Bismut and Skalli~\cite[Theorem~I.3]{bs77}, which says that a stopping time defined in terms of the Snell envelope of the regular process $R$ is optimal. Their proof assumes bounded reward $R$ but they note on page~301 that it actually suffices that $R$ be of class $(D)$. The proof of Bismut and Skalli builds on the (nontrivial) existence of a Snell envelope and further limiting arguments involving sequences of stopping times. In contrast, our proof is based on elementary functional analytic arguments in the Banach space setting of Theorem~\ref{thm:reg1}, which is of independent interest.

Note that $x$ solves the relaxed optimal stopping problem if and only if $R$ is {\em normal} to $\C$ at $x$, i.e.\ if $R\in\partial\delta_\C(x)$ or equivalently $x\in\partial\sigma_\C(R)$, where
\[
\sigma_\C(R) = \sup_{x\in\C}\langle R,Dx\rangle.
\]
Here, $\partial$ denotes the {\em subdifferential} of a function; see e.g.\ \cite{roc74}. If $R$ is nonnegative, we have $\sigma_\C(R)=\|R\|_{\R^1}$ (by Krein--Milman) and the optimal solutions of the relaxed stopping problem are simply the subgradients of the $\R^1$-norm at $R$.

\section{Cadlag processes}\label{sec:cadlag}

This section extends the previous section to optimal quasi-stopping problems when the reward process $R$ is merely {\em cadlag and of class $(D)$}. In this case, optimal stopping times need not exist (see the discussion on page 1) but we will prove the existence of a quasi-stopping time by functional analytic arguments analogous to those in Section~\ref{sec:regular}.

The Banach space of cadlag functions equipped with the supremum norm will be denoted by $D$. The space of purely discontinuous Borel measures will be denoted by $\tilde M$. The dual of $D$ can be identified with $M\times\tilde M$ through the bilinear form
\[
\langle y,(u,\tilde u)\rangle := \int ydu + \int y_-d\tilde u
\]
and the dual norm is given by 
\[
\sup_{y\in D}\left\{\left.\int ydu + \int y_-d\tilde u\,\right|\,\|y\|\le 1\right\}=\|u\|+\|\tilde u\|,
\]
where $\|u\|$ denotes the total variation norm on $M$. This can be deduced from \cite[Theorem~1]{pes95} or seen as the deterministic special case of \cite[Theorem~VII.65]{dm82} combined with \cite[Remark~VII.4(a)]{dm82}.

The following result from \cite{pp18b} provides the functional analytic setting for analyzing quasi-stopping problems with cadlag processes of class $(D)$.

\begin{theorem}\label{thm:L1}
The space $\D^1$ of optional cadlag processes of class $(D)$ equipped with the norm
\[
\|y\|_{\D^1}:=\sup_{\tau\in\T}E|y_\tau|
\]
is Banach and its dual can be identified with 
\[
\hat\M^\infty:=\{(u,\tilde u)\in L^\infty(M\times\tilde M)\mid u\text{ is optional},\, \tilde u\text{ is predictable}\}
\]
through the bilinear form
\[
\langle y,(u,\tilde u)\rangle =E\left[\int ydu+\int y_-d\tilde u\right].
\]
The optional projection is a continuous surjection of $L^1(D)$ to $\D^1$ and its adjoint is the embedding of $\hat\M^\infty$ to $L^\infty(M\times \tilde M)$. The norm of $\D^1$ is equivalent to
\[
 p(y):= \inf_{z\in L^1(D)} \{E\|z\| \mid \op z = y\},
\]
which has the dual representation
\[
p(y)=\sup\{\langle y,(u,\tilde u)\rangle\,|\,\esssup(\|u\|+\|\tilde u\|)\le 1\}.
\]
\end{theorem}

The space $M\times\tilde M$ may be identified with the space $\hat X_0$ of (not necessarily left-continuous) functions $x:\reals_+\to\reals$ of bounded variation which are constant on $(T,\infty]$ and have $x_0=0$. Indeed, every $x\in\hat X_0$ can be written uniquely as
\[
x_t = Dx([0,t)) + \tilde Dx([0,t]),
\]
where $\tilde Dx\in\tilde M$ and $Dx\in M$ are the measures associated with the functions $\tilde x_t :=\sum_{s\le t} (x_s-x_{s-})$ and $x-\tilde x$, respectively. The linear mapping $x\mapsto(Dx,\tilde Dx)$ defines an isomorphism between $\hat X_0$ and $M\times\tilde M$. The value of $x$ for $t>T$ will be denoted by $x_{T+}$. Similarly, the space $\hat \M^\infty$  may be identified with the space $\hat\N_0^\infty$ of predictable processes $x$ with $x\in \hat X_0$ almost surely and $(Dx,\tilde Dx)\in\hat\M^\infty$.


Problem \eqref{oqs} can be written as
\[
\maximize\quad \langle R, (Dx,\tilde Dx)\rangle\quad\ovr\quad x\in\hat\C_e,
\]
where 
\[
\hat\C_e:=\{x\in\hat\N_0^\infty\,|\, (Dx,\tilde Dx)\in\hat\M^\infty_+,\ x_t\in\{0,1\}\}.
\]
Indeed, the equations $\tau(\omega) = \inf\{t\in\reals\mid x_t(\omega)\ge 1\}$ and $\tilde\tau(\omega) = \inf\{t\in\reals\mid x_t-x_{t-}(\omega)\ge 1\}$ give a one-to-one correspondence between the elements of $\hat\T$ and $\hat\C_e$.

Consider also the convex relaxation
\begin{equation}\label{Dcr}\tag{ROQS}
\maximize\quad \langle R,(Dx,\tilde Dx)\rangle\quad\ovr\quad x\in\hat\C,
\end{equation}
where 
\[
\hat\C:=\{x\in\hat\N_0^\infty\,|\, (Dx,\tilde Dx)\in\hat\M^\infty_+,\ x_{T+}\le 1\}.
\]

\begin{lemma}\label{lem:cadkm}
The set $\hat\C$ is convex, $\sigma(\hat \M^\infty,\D^1)$-compact and the set of quasi-stopping times $\hat \C_e$ is its extreme points. Moreover, the set of stopping times is $\sigma(\hat \M^\infty,\D^1)$-dense in $\hat \C_e$ and, thus, $\C$ is $\sigma(\hat \M^\infty,\D^1)$-dense in $\hat\C$.
\end{lemma}

\begin{proof}
The set $\hat\C$ is a closed convex set of the unit ball that $\hat\N_0^\infty$ has as the dual of the Banach space $\D^1$. The compactness thus follows from Banach-Alaoglu. It is easily shown that the elements of $\hat\C_e$ are extreme points of $\hat\C$.

If $x\notin\hat\C_e$, there exist $\bar s\in(0,1)$ such that
\begin{align*}
x^1_t &:=\frac{1}{\bar s }[x_t\wedge\bar s],\quad\quad x^2_t:=\frac{1}{1-\bar s}[(x_t-\bar s)\vee 0]
\end{align*}
are distinguishable processes that belong to $\hat\C$. Since $x=\bar s x^1+(1-\bar s)x^2$, $x$ is not an extremal in $\hat C$.

To prove the last claim, let $(\tau,\tilde\tau)$ be a quasi-stopping time and $(\tau^\nu)$ an announcing sequence for $\tilde\tau$. We then have
\[
\langle(\delta_{\tau\wedge\tau^\nu},0),y\rangle\to\langle(\delta_\tau,\delta_{\tilde\tau}),y\rangle
\]
for every $y\in\D^1$.
\end{proof}

Just like in Section~\ref{sec:regular}, a combination of Lemma~\ref{lem:cadkm} and Theorem~\ref{thm:bauer} gives the following existence result which was established in Bismut~\cite{bis79} using more elaborate techniques based on the existence of Snell envelopes.

\begin{theorem}\label{thm:os}
If $R\in\D^1$, then optimal quasi-stopping time in \eqref{oqs} exists and the optimal values of \eqref{oqs}, \eqref{os} and \eqref{Dcr} are all equal.
\end{theorem}

As another implication of Lemma~\ref{lem:cadkm} and Theorem~\ref{thm:L1}, we recover the following result of Bismut which says that the seminorms in Theorem~\ref{thm:L1} are not just equivalent but equal.

\begin{theorem}[{\cite[Theorem~4]{bis78}}]\label{thm:equiv1}
For every $y\in\D^1$,
\[
\|y\|_{\D^1}=\inf_{z\in L^1(D)}\{E\|z\|_D\mid \op z = y\}.
\]
\end{theorem} 

\begin{proof}
The expression on the right is the seminorm $p$ in Theorem~\ref{thm:L1} with the dual representation 
\[
p(y)=p(|y|)=\sup_{x\in\hat\C}\langle |y|,(Dx,\tilde Dx)\rangle
\]
which, by Theorem~\ref{thm:os}, equals the left side.
\end{proof}

Combining the above with Theorem~\ref{thm:reg1} gives a simple proof of the following.

\begin{theorem}[{\cite[Theorem~3]{bis78}}]\label{thm:equivreg}
For every $y\in\R^1$,
\[
\|y\|_{\R^1}=\inf_{z\in L^1(C)}\{E\|z\|_D\mid \op z = y\}.
\]
\end{theorem} 

\begin{proof}
By Jensen's inequality, the left side is less than the right which is the seminorm $p$ in Theorem~\ref{thm:reg1} with the dual representation
\begin{align*}
  p(y) &= \sup\{\langle y,u\rangle\,|\,\esssup(\|u\|)\le 1\}\\
  &\le \sup\{\langle y,(u,\tilde u)\rangle\,|\,\esssup(\|u\|+\|\tilde u\|)\le 1\}\\
  &=\sup_{x\in\hat\C}\langle |y|,(Dx,\tilde Dx)\rangle,
\end{align*}
which, again by Theorem~\ref{thm:os}, equals the left side.
\end{proof}

\section{Non-cadlag processes}\label{sec:usc}

This section gives a further extension to cases where the reward process is not necessarily cadlag but merely {\em right-upper semicontinuous} (right-usc) in the sense that $R\ge\cev R$. In this case, the objective of the relaxed quasi-optimal stopping problem \eqref{Dcr} need not be continuous. The following lemma says that it is, nevertheless, upper semicontinuous, so Bauer's maximum principle still applies.

\begin{lemma}\label{lem:usc}
If $R$ is right-usc and of class $(D)$, then the functional 
\begin{align*}
\hat\J(u,\tilde u)=\begin{cases} 
E\left[\int Rdu + \int \vec R d\tilde u\right]\quad &\text{if }(u,\tilde u)\in\hat\M^\infty_+\\
-\infty\quad&\text{otherwise}
\end{cases}
\end{align*}
is $\sigma(\hat\M^\infty,\D^1)$-usc.
\end{lemma}

\begin{proof}
Recalling that every optional process of class $(D)$ has a majorant in $\D^1$ (see \cite[Remark 25, Appendix I]{dm82}), the first example in  \cite[Section~8]{pt18a} shows, with obvious changes of signs, that $\hat\J$ is usc. 
\end{proof}

Combining Lemma~\ref{lem:usc} with Theorem~\ref{thm:bauer} gives the existence of a relaxed quasi-stopping time at an extreme point of $\C$ which, by Lemma~\ref{lem:cadkm}, is a quasi-stopping time. We thus obtain the following.


\begin{theorem}\label{thm:usc}
If $R$ is right-usc and of class $(D)$, then \eqref{oqs} has a solution. 
\end{theorem}

We have not been able find the above result in the literature but it can be derived from Theorem~2.39 of El Karoui~\cite{elk81} on ``divided stopping times'' (temps d'arret divis\'es). A recent analysis of divided stopping times can be found in Bank and Besslich \cite{bb18a}. These works extend Bismut's approach on optimal quasi-stopping by dropping the assumption of right-continuity and augmenting quasi-stopping times with a third component that acts on the right limit of the reward process. Much like Bismut's approach, \cite{elk81,bb18a} build on the existence of a Snell envelope.

Theorem~\ref{thm:usc} yields the existence of an optimal stopping time when the reward process $R$ is {\em subregular} in the sense that it is right-usc, of class $(D)$ and $\vec R\le\pp R$.



\begin{theorem}\label{thm:uscreg}
If $R$ is subregular, then \eqref{os} has a solution and its optimum value equals that of \eqref{oqs}.
\end{theorem}

\begin{proof}
Clearly, the optimum value of \eqref{oqs} is at least that of \eqref{os} while for subregular $R$,
\[
E[R_\tau+\vec R_{\tilde\tau}] \le E[R_\tau+\pp R_{\tilde\tau}] = E[R_\tau+R_{\tilde\tau}] = ER_{\tau\wedge\tilde\tau},
\]
where the first equality holds by the definition of predictable projection. The claim now follows from Theorem~\ref{thm:usc}.
\end{proof}

The above seems to have been first established in Bismut and Skalli~\cite[Section~II]{bs77} for bounded $R$ (again, they mention on page 301 that, instead of boundedness, it would suffice to assume that $R$ is of class $(D)$).

Regularity properties are preserved under compositions with convex functions much like martingale properties. Indeed, if $R$ is regular and $g$ is a real-valued convex function on $\reals$ then $g(R)$ is subregular as soon as it is of class $(D)$. Indeed, for any $\tau\in\T_p$, conditional Jensen's inequality gives
\[
E[g(\vec R_\tau)\one_{\tau<+\infty}]= E[g(\pp R_\tau)\one_{\tau<+\infty}] \le E[g(R_\tau)\one_{\tau<+\infty}].
\]
Similarly, if $R$ is subregular and $g$ is a real-valued increasing convex function, then $g(R)$ is subregular as soon as the composition is of class $(D)$.


\section{Duality}

We end this paper by giving optimality conditions and a dual problem for the optimal stopping problems. The derivations are based on the conjugate duality framework of~\cite{roc74} which addresses convex optimization in general locally convex vector spaces. 
The results below establish the existence of dual solutions without assuming the existence of optimal (quasi-)stopping times. They hold without any path properties as long as the reward process $R$ is of class $(D)$.

We denote the space of martingales of class $(D)$ by $\R^1_m$.

\begin{theorem}\label{thm:osdual}
Let $R$ be of class $(D)$. Then the optimum values of \eqref{oqs} and \eqref{os} coincide and equal that of
\begin{equation}\label{d}\tag{DOS}
\inf\{EM_0 \mid M\in\R^1_m,\ R\le M\},
\end{equation}
where the infimum is attained.

Moreover, $x\in\hat\C$ is optimal in the convex relaxation of \eqref{oqs} if and only if there exists $M\in\R^1_m$ with $R\le M$ and
\begin{align}
\int (M-R)dx+\int(M_--\vec R)d\tilde x &= 0,\label{eq:oc1}\\ 
x_{T+}=1\quad\text{or}\quad M_T&=0\label{eq:oc2}
\end{align}
almost surely. Thus, $(\tau,\tilde\tau)\in\hat\T$ is optimal in \eqref{oqs} if and only if there exists $M\in\R^1_m$ with $R\le M$, $M_\tau=R_\tau$, $M_{\tilde \tau_-} =\vec R_{\tilde\tau}$ and almost surely either $\tau+\tilde\tau<\infty+$ or $M_T=0$.

In particular, $x\in\C$ is optimal  in the convex relaxation of \eqref{os} if and only if there exists $M\in\R^1_m$ with $R\le M$ and
\begin{align*}
\int (M-R)dx&= 0,\\ 
x_{T+}=1\quad\text{or}\quad M_T&=0
\end{align*}
almost surely. Thus, $\tau\in\T$ is optimal in \eqref{os} if and only if there exists $M\in\R^1_m$ with $R\le M$, $M_\tau=R_\tau$ and almost surely either $\tau<\infty+$ or $M_T=0$.
\end{theorem}

\begin{proof}
By \cite[Remark~25, Appendix~I]{dm82}, there are measurable processes $z$ and $\tilde z$ such that $R =\op z$, $\vec R=\op{\tilde z}$ and $E[\sup_tz_t+\sup_t\tilde z_t]<\infty$. The optimum value and optimal solutions of \eqref{oqs} coincide with those of
\begin{align}\label{eq:ros}
&\maximize_{x\in\hat\N^\infty}\quad E\left[\hat \J(Dx,\tilde Dx)   - \rho(x_{T+}-1)^+\right],
\end{align}
where $\rho:=\sup_tz_t+\sup_t\tilde z_t+1$ and $\hat\J$ is defined as in Lemma~\ref{lem:usc}. Indeed, if $x$ is feasible in \eqref{eq:ros} then $\bar x:=x\wedge 1$ is feasible in \eqref{oqs} and since $x-\bar x$ is an increasing process with $(x-\bar x)_{T+}=(x_{T+}-1)^+$, we get 
\begin{align*}
\hat\J(D\bar x,\tilde D\bar x) &= \hat\J(Dx,\tilde Dx) - \hat\J(D(x-\bar x),\tilde D(x-\bar x))\\
&\ge \hat\J(Dx,\tilde Dx) - E\rho(x_{T+}-1)^+.
\end{align*}

Problem~\eqref{eq:ros} fits the general conjugate duality framework of \cite{roc74} with $U=L^\infty$, $Y=L^1$ and
\[
F(x,w) = -\hat\J(Dx,\tilde Dx) + E\rho(x_{T+}+w-1)^+.
\]
By \cite[Theorem~22]{roc74}, $w\to F(0,w)$ is continuous on $L^\infty$ in the Mackey topology that it has as the dual of $L^1$. Thus, by \cite[Theorem~17]{roc74}, the optimum value of \eqref{eq:ros} coincides with the infimum of the dual objective 
\[
g(y) :=-\inf_{x\in\hat\N^\infty} L(x,y),
\]
where $L(x,y) :=\inf_{w\in L^\infty}\{F(x,w)-Ewy\}$, and moreover, the infimum of $g$ is attained. By the interchange rule \cite[Theorem~14.60]{rw98},
\begin{align*}
L(x,y)&=
\begin{cases}
+\infty & \text{if $x\notin\hat\N^\infty_+$},\\
-\hat\J(Dx,\tilde Dx) + E\left[\inf_{u\in\reals}\{\rho(x_{T+}+u-1)^+ - uy\}\right]&\text{otherwise}\\
\end{cases}\\
&=
\begin{cases}
+\infty & \text{if $x\notin\hat\N^\infty_+$},\\
-\hat\J(Dx,\tilde Dx)+ E\left[x_{T+}y-y - \delta_{[0,\rho]}(y)\right]&\text{otherwise}.
\end{cases}
\end{align*}
We have
\[
E[x_{T+}y] = E[\int(y\one)dx+\int(y\one)d\tilde x] = \langle M,(Dx,\tilde Dx)\rangle,
\]
where $M=\op(y\one)\in\R^1_m$. Thus,
\[
L(x,y) = 
\begin{cases}
+\infty & \text{if $x\notin\hat\N^\infty_+$},\\
-\hat\J(Dx,\tilde Dx) + \langle M,(Dx,\tilde Dx)\rangle - EM_T & \text{if $x\in\hat\N^\infty_+$ and $0\le M_T\le\rho$,}\\
-\infty & \text{otherwise}.
\end{cases}
\]
The dual objective can be written as
\begin{align*}
g(y) &= 
\begin{cases}
EM_0 & \text{if $0\le M_T\le\rho$, $M\ge R$ and $M_-\ge\vec R$},\\
+\infty & \text{otherwise}.
\end{cases}
\end{align*}
Since $M$ is cadlag, $M_-\ge\vec R$ holds automatically when $M\ge R$. In summary, the optimum value of \eqref{oqs} equals that of \eqref{d}.

The dual problem of \eqref{os} is obtained similarly by defining
\[
F(x,w) = -\J(Dx) + E\rho(x_{T+}+w-1)^+.
\]
The function $w\to F(0,w)$ is again Mackey-continuous on $L^\infty$ and one finds that the dual is again \eqref{d}. Thus, the optimum value of \eqref{os} equals that of \eqref{d}.

As to the optimality conditions, \cite[Theorem~15]{roc74} says that $x$ is optimal in \eqref{eq:ros} and $y$ is optimal in the dual if and only if
\[
0\in\partial_x L(x,y),\quad 0\in\partial_y[-L](x,y).
\]
The former means that $x\in\hat\N^\infty_+$, $M\ge R$ and 
\[
\int (M-R)dx=0, \quad \int (M_--\vec R)d\tilde x=0\quad P\text{-a.s.}
\]
By the interchange rule for subdifferentials (\cite[Theorem~21c]{roc74}), the latter is equivalent to \eqref{eq:oc2}.
\end{proof}


Note that for any martingale $M\in\R^1_m$,
\[
\sup_{\tau\in\T}ER_\tau = \sup_{\tau\in\T}E(R_\tau+M_T-M_{\tau})\le E\sup_{t\in[0,T]}(R_t+M_T-M_t),
\]
where the last expression is dominated by $EM_0$ if $R\le M$. Thus,
\begin{align*}
\sup_{\tau\in\T}ER_\tau &\le \inf_{M\in\R^1_m}E\sup_{t\in[0,T]}(R_t+M_T-M_t)\\
&\le \inf_{M\in\R^1_m}\{E\sup_{t\in[0,T]}(R_t+M_T-M_t)\,|\,R\le M\}\\
&\le \inf_{M\in\R^1_m}\{EM_0\,|\,R\le M\},
\end{align*}
where, by Theorem~\ref{thm:os}, the last expression equals the first one as soon as $R$ is of class $(D)$. The optimum value of the stopping problem then equals
\[
\inf_{M\in\R^1_m}E\sup_{t\in[0,T]}(R_t+M_T-M_t).
\]
This is the dual problem derived in Davis and Karatzas~\cite{dk94} and Rogers~\cite{rog2}. Note also that if $Y$ is the Snell envelope of $R$ (the smallest supermartingale that dominates $R$), then the martingale part $M$ in the Doob--Meyer decomposition $Y=M-A$ is dual optimal. These facts were obtained in \cite{dk94} and \cite{rog2} under the assumptions that $\sup_tR_t$ is integrable. 

\bibliographystyle{plain}
\bibliography{sp}

\end{document}